\documentclass{amsart}

\usepackage{latexsym}
\usepackage{amsmath}
\usepackage{amssymb}
\usepackage{amsthm}
\usepackage{bm}
\usepackage{esint}
\usepackage[inline]{enumitem}
\theoremstyle{plain}
\newtheorem{theorem}{Theorem}[section]
\newtheorem{corollary}[theorem]{Corollary}
\newtheorem{lemma}[theorem]{Lemma}

\theoremstyle{definition}

\newtheorem{definition}[theorem]{Definition}

\theoremstyle{remark}
\newtheorem{example}[theorem]{Example}

\numberwithin{theorem}{section}
\numberwithin{equation}{section}

\newcommand{\Z}{\mathbb{Z}}
\newcommand{\R}{\mathbb{R}}
\newcommand{\diam}{\mathrm{diam}}
\newcommand{\inrad}{\mathrm{inrad} \, }
\newcommand{\cl}{\overline}

\newcommand{\loc}{\mathrm{loc}}

\DeclareMathOperator{\divergence}{div}
\newcommand{\laplacian}{\Delta}
\DeclareMathOperator*{\osc}{osc}
\DeclareMathOperator*{\esssup}{ess\,sup}
\DeclareMathOperator*{\essinf}{ess\,inf}

\newcommand{\capacity}{\mathrm{cap}}
\newcommand{\wkto}{\rightharpoonup}

\newcommand{\trinorm}[1]
{{
    \left\vert\kern-0.20ex\left\vert\kern-0.20ex\left\vert
    #1 
    \right\vert\kern-0.20ex\right\vert\kern-0.20ex\right\vert
}}

\begin{document}

%%%%%%%%%%%%%%%%%%%%%%%%%%%%%%%%%%%%%%%%
% Author's name, Title, Date and other information
%%%%%%%%%%%%%%%%%%%%%%%%%%%%%%%%%%%%%%%%

\title[Global H\"{o}lder Solvability of parabolic equations]
{Global H\"{o}lder Solvability of parabolic equations on domains with capacity density conditions}
\author{Takanobu Hara}
\email{takanobu.hara.math@gmail.com}
\address{Tohoku University, Mathematical Institute, 6-3, Aramaki Aza-Aoba, Aoba-ku, Sendai, Miyagi 980-8578, Japan}
\date{\today}
\subjclass[2020]{Primary: 35K20, Secondary: 35B65, 35R05, 31B15} 

%%%%%%%%%%%%%%%%%%%%%%%%%%%%%%%%%%%%%%%%
% Abstract
%%%%%%%%%%%%%%%%%%%%%%%%%%%%%%%%%%%%%%%%

\begin{abstract}
We investigate the Cauchy-Dirichlet problem for linear parabolic equations in divergence form. 
Under mild assumptions on the source term and the domain, we prove the existence of globally H\"{o}lder continuous solutions. 
Notably, our results accommodate data exhibiting singularities nearly as critical as the inverse square of the distance from the boundary.
\end{abstract}

%%%%%%%%%%%%%%%%%%%%%%%%%%%%%%%%%%%%%%%%
% Title & Table
%%%%%%%%%%%%%%%%%%%%%%%%%%%%%%%%%%%%%%%%

\maketitle
% \setcounter{tocdepth}{1}
%\tableofcontents

%%%%%%%%%%%%%%%%%%%%%%%%%%%%%%%%%%%%%%%%
% Body
%%%%%%%%%%%%%%%%%%%%%%%%%%%%%%%%%%%%%%%%

\section{Introduction}

Consider the Cauchy-Dirichlet problem for parabolic equations of the form
\begin{align}
\mathcal{H} u := \frac{ \partial u }{ \partial t } - \divergence \left( A \nabla u \right) = f \quad & \text{in} \ D \times (0, T),
\label{eqn:parabolic}
\\
u = 0  \quad & \text{on} \ (\partial D) \times (0, T),
\label{eqn:BV}
\\
u(\cdot, 0) = 0  \quad & \text{in} \ D.
\label{eqn:IV}
\end{align}
Here, $D$ is an open set in $\R^{n}$ ($n \ge 2$),
$0 < T \le \infty$, and $\nabla u$ denotes the gradient of $u$ in $\R^{n}$.
The coefficient matrix $A = A(x, t)$ is a matrix-valued function in $L^{\infty}(D \times (0, T))^{n \times n}$ satisfying the uniform ellipticity and boundedness conditions
\begin{equation}\label{eqn:elliptic}
A(x, t) \xi \cdot \xi \ge |\xi|^{2}, \quad |A(x, t) \xi| \le L |\xi|  \quad \forall \xi \in \R^{n}, \ \forall (x, t) \in D \times (0, T)
\end{equation}
for some $1 \le L < \infty$.
Throughout the paper, do not impose any further regularity assumptions on $A$.
The source term $f$ is a function in $L^{1}_{\loc}(D \times (0, T))$.

We give an existence result of a globally H\"{o}lder continuous weak solution to \eqref{eqn:parabolic}-\eqref{eqn:IV}
under mild conditions on $f$ and $D$.
This result covers data whose singularity is nearly of the order of the inverse square of the distance to the boundary.

\subsection{Background}

The solvability of \eqref{eqn:parabolic} in H\"{o}lder spaces has been widely studied and remains a fundamental topic in the field. 
A comprehensive account of the regularity theory for such equations was established by Ladyzhenskaya, Solonnikov, and Ural'tseva in their classic monograph \cite{MR241822}.
Specifically, if
\begin{equation}\label{eqn:LSU}
f \in L^{p}( 0, T; L^{q}(D)) \quad \text{with} \quad \frac{n}{2q} + \frac{1}{p} < 1,
\end{equation}
then $u$ is locally H\"{o}lder continuous in $D \times (0, T)$. This condition has become standard in subsequent literature,
including \cite{MR226168, MR1465184}.
Regarding boundary regularity, it has been established that interior estimates extend up to a boundary point, provided that the complement of the spatial domain $D$ satisfies a suitable volume density condition at that point. Furthermore, it is well known that the same conclusion remains valid under the weaker assumption of the capacity of the complement in space-time (see, e.g., \cite{MR399659, MR563383, MR654859, MR672714, MR951629}).
These results imply that if a capacity density condition is satisfied uniformly at all boundary points, it follows that the solution is globally H\"{o}lder continuous.

Is the aforementioned story -- establishing local estimates and then extending them to the boundary -- truly absolute?
Let us reconsider this from a naive perspective.
It is natural to expect that good regularity estimates hold as long as the Dirichlet boundary condition \eqref{eqn:BV} is satisfied on a sufficiently large portion of the boundary.
In fact, the solution to \eqref{eqn:parabolic}-\eqref{eqn:IV} can be represented using
the Dirichlet heat kernel $p_{D}(x, y, t, s)$ as follows:
\begin{equation}\label{eqn:heat_kernel}
u(x, t)= \int_{0}^{T} \int_{D} p_D (x, y, t, s) f(y, s) \, dyds.
\end{equation}
For each $(x, t) \in D \times (0, T)$, $p_{D}(x, \cdot, t, \cdot)$ decays near the lateral boundary of $D \times (0, T)$ because it also satisfies a parabolic equation.
Therefore, the assumptions imposed on the boundary behavior of $f$ should be weaker than those required in the interior.
However, the standard approach ignores this additional regularity effect. 
A closer inspection of the standard proof reveals that
when deriving estimates at each boundary point, the assumption that the exterior of the domain is uniformly secured is not utilized.
Consequently, the standard approach only permits singularities of the form $f(x, t) = d_{\R^{n}}(x, \xi)^{\epsilon - 2}$ for a fixed $\xi \in \partial D$ (peak-like behavior)
and cannot treat $f$ with critical boundary behavior, such as $f(x, t) = d_{\R^{n}}(x, \partial D)^{\epsilon - 2}$ (wall-like behavior). 
Thus, the standard boundary regularity theory essentially proves itself to be non-sharp for typical domains, such as polygons.
% Do not believe the book; understand what is written.

On the other hand, although a direct application of \eqref{eqn:heat_kernel} merits consideration if $D$ has a smooth boundary, 
obtaining concrete conditions for $f$ becomes a non-trivial task if $D$ is complicated.
Pointwise estimates of the heat kernel on bounded domains have naturally been studied, but they impose strong assumptions on domains
(see, e.g., \cite{MR879702, MR1900329, MR3220451}).
Indeed, near the corners of a polygon, the boundary behavior of harmonic functions is not comparable to $d_{\R^{n}}(x, \partial D)$,
and no quantitative criterion for estimating $p_{D}$ can be found.
This approach is incompatible with the fact that boundary regularity can be established solely through assumptions on the exterior of the domain.
Should we really abandon the possibility of producing different results using the standard tools in regularity theory?

The present paper aims to fill this gap by constructing supersolutions to \eqref{eqn:parabolic} following the spirit of Ancona's method.
In \cite{MR856511}, Ancona constructed a supersolution to elliptic equations by gluing infinitely many harmonic measures. 
Subsequent developments have significantly advanced analysis on domains satisfying the capacity density condition 
(see, e.g., \cite{MR946438, MR1010807, MR1924196, haraCDC}), 
yet Ancona's contribution does not appear to have been fully appreciated. 
The author has recently pursued extensions and applications of this approach in related works \cite{MR4791992, haraholder, MR4971571}.
In this manuscript, we provide an extension to parabolic equations that is largely faithful to the original argument presented in \cite{MR856511}.

\subsection{Results}

Our main results establish the global regularity theory for parabolic equations by constructing a sharp barrier function tailored to domains with rough boundaries.
In Theorem \ref{thm:main}, we construct a supersolution to \eqref{eqn:parabolic}, which serves as a key barrier function.
By combining this barrier with the comparison principle,
we establish the existence and global H\"{o}lder regularity of weak solutions to
the Cauchy-Dirichlet problem \eqref{eqn:parabolic}-\eqref{eqn:IV} on domains satisfying the capacity density condition \eqref{eqn:CDC}
(Theorem \ref{thm:main2}).
A notable feature of this result is that it accommodates external forces with strong singularities, specifically allowing data that scales like the inverse square of the distance to the boundary.
The final result concerning the non-homogeneous boundary value problem is provided in Thoerem \ref{thm:main3}.
Although Theorem \ref{thm:main4} may be a known result, the author is unaware of any literature containing this statement.

\subsection*{Structure of the paper}

The remainder of this paper is organized as follows.
In Section \ref{sec:preliminaries}, we present standard results on parabolic equations needed for the subsequent discussion.
In Section \ref{sec:barrier}, we construct the desired barriers. This part is the main contribution of this paper. 
In Section \ref{sec:existence}, we discuss an existence theorem for H\"{o}lder continuous solutions. 
In Section \ref{sec:boundary_value_problem} extends the results of Section \ref{sec:existence} to handle boundary value problems. 
Finally, in Section \ref{sec:boundary_regularity}, we provide supplementary explanations for parts of Section \ref{sec:preliminaries} that require further clarification.

\subsection*{Notation}

To deal with parabolic equations, we introduce the following terminology.

\begin{itemize}
\item
For $X = (x, t), Y = (y, s) \in D \times \R$,
the parabolic distance $d_{p}(X, Y)$ is
\[
d_{p}(X, Y) = \max \left\{ d_{\R^{n}}(x, y), |t -s|^{1 / 2} \right \},
\]
where $d_{\R^{n}}(\cdot, \cdot)$ is the Euclidean distance in $\R^{n}$.
\item
For $X = (x, t) \in \R^{n} \times \R$ and $R > 0$,
the parabolic cylinder $Q(X, R)$ is defined by
\[
Q(X, R)
:=
B(x, R) \times (t - R^{2}, t).
\]
\item
Let $D \subset \R^{n}$ be an open set, and let $0 < T \le \infty$.
For the cylinder $D \times (0, T)$, we define a function $\delta_{\Gamma}$ on $D \times (0, T)$ by
\begin{equation}
\delta_{\Gamma}(X)
=
\inf_{Y \in \Gamma \times I \cup D \times \{ 0 \} } d_{p}(X, Y).
\end{equation}
\end{itemize}

Let $D$ be an open set in $\R^{n}$.
Denote by $H^{1}(D)$ the set of all weakly differentiable functions on $D$ such that
\[
\| u \|_{H^{1}(D)} 
:=
\left(
\int_{D} |u|^{2} \, dx
+
\int_{D} |\nabla u|^{2} \, dx
\right)^{1 / 2}
\]
is finite.
For $u \in H^{1}(D)$, we define
\begin{equation}\label{eqn:sobolev_bv}
\sup_{\partial D} u := \inf \left\{ k \in \R \colon (u - k)_{+} \in H_{0}^{1}(D) \right\},
\quad
\inf_{\partial D} u = -(\sup_{\partial D} - u).
\end{equation}

\section{Assumptions on $D$: Capacity density condition}\label{sec:preliminaries}

Here, we present preliminaries concerning the spatial domain $D$. 
Let $D \subset \R^{n}$ be an open set, and let $\Gamma$ be a nonempty closed subset of $\partial D$.
We consider the condition
\begin{equation}\label{eqn:CDC}
\exists \gamma > 0 \quad \text{s.t.} \quad
\frac{ \capacity( \cl{B(\xi, R)} \setminus D, B(\xi, 2R)) }{ \capacity( \cl{B(\xi, R)}, B(\xi, 2R)) } \ge \gamma
\quad \forall R > 0, \ \forall \xi \in \Gamma.
\end{equation}
Here, for the pair of an open set $U \subset \R^{n}$ and its compact subset $K \subset U$,
the variational capacity $\capacity(K, U)$ is defined by
\begin{equation}\label{eqn:variational_capacity}
\capacity(K, U)
:=
\inf
\left\{
\int_{\R^{n}} |\nabla u|^{2} \, dx \colon u \in C_{c}^{\infty}(U), \ u \ge 1 \, \text{on} \, K
\right\}.
\end{equation}
Bounded Lipschitz domains are typical examples of domains satisfying the CDC. The corkscrew condition is also a useful geometric sufficient condition. 
Furthermore, criteria based on Hausdorff measure are well known; for details, see \cite{MR2305115, MR4306765}. 
Below, we mention some examples that do not necessarily fall into these standard categories.

\begin{example}
When $n = 2$, any simply connected bounded open set $D \subset \R^{2}$ satisfies \eqref{eqn:CDC} with $\Gamma = \partial D$.
\end{example}

\section{Parabolic equations}

To consider parabolic equations, we introduce space-time functions spaces.
For an open set $D \subset \R^{n}$ and an open interval $I \subset \R$, we define
\[
V( D \times I )
:=
L^{\infty}(I ; L^{2}(D)) \cap L^{2}( I ; H^{1}(D)).
\]
The corresponding local space $V_{\loc}( D \times I )$ is defined in the usual manner.

We define local weak solutions to \eqref{eqn:parabolic} and discuss some regularity results for them.

\begin{definition}\label{def:weak_sol}
Let $u \in V_{\loc}( D \times I )$, and let $f \in L^{1}_{\loc}(D \times I)$.
We say that $u$ is a \textit{local weak supersolution (subsolution)} to $\mathcal{H} u = f$ in $D \times I$ if
\begin{equation}\label{eqn:wf1}
\iint_{D \times I} 
- u \frac{ \partial \varphi }{ \partial t }
+
A \nabla u \cdot \nabla \varphi 
\, dx dt
\ge (\le)
\iint_{D \times I} \varphi f \, dx dt
\end{equation}
for all $\varphi \in C_{c}^{\infty}(D \times I)$.
\end{definition}

Let us introduce the Steklov average of space-time functions.

\begin{definition}
Let $u \in L^{1}(D \times (t_{1}, t_{2}))$, and let $0 < h < t_{2} - t_{1}$.
For $t \in (t_{1}, t_{2} - h)$, we define the \textit{Steklov average} $u_{h}$ of $u$ by
\[
u_{h}
:=
\frac{1}{h} \int_{t}^{t + h} u(\cdot, \tau) \, d \tau.
\]
\end{definition}

If $u$ is a supersolution (subsolution) to $\mathcal{H} u = f$ in $D \times (t_{1}, t_{2})$ in the sense of Definition \ref{def:weak_sol}, then, we have
\begin{equation}\label{eqn:wf2}
\int_{ D \times \{ t \} }
\frac{ \partial u_{h} }{ \partial t } \varphi + [A \nabla u]_{h} \cdot \nabla \varphi \, dx
\ge (\le)
\int_{ D \times \{ t \} } \varphi f_{h} \, dx
\end{equation}
for any $t_{1} < t < t_{2} - h$ and $\varphi \in C_{c}^{\infty}(D)$.
Conversely, if $u \in V_{\loc}(D \times (t_{1}, t_{2}))$ satisfies \eqref{eqn:wf2} for all $\varphi \in C_{c}^{\infty}(D)$, $t_{1} < t < t_{2}$, and small $h$,
then, $u$ is a supersolution (subsolution) to $\mathcal{H} u = f$ in the sense of Definition \ref{def:weak_sol}.
In fact, by \cite[p.82]{MR241822}, if $f \in L^{r}(t_{1}, t_{2}; L^{q}(D) )$ and $0 < \epsilon < t_{2} - t_{1}$,
then, $f_{h} \to f$ in $L^{r}(t_{1}, t_{2} - \epsilon; L^{q}(D))$
as $h \to 0$. 

Substituting $\varphi = \epsilon / (( u_{h} - v_{h} )_{+} + \epsilon)  \eta$, ($\eta \in C_{c}^{\infty}(D \times I)$, $\eta \ge 0$) into \eqref{eqn:wf2},
and passing to the limit $h \to 0$,
we get the following energy estimate.

\begin{lemma}\label{lem:gluing}
Assume that $u$ and $v$ are weak supersolutions to $\mathcal{H} u = f$ in $D \times I$.
Then, $\min \{ u, v \}$ is a supersolution to the same equation.
\end{lemma}

Substituting $\varphi = u_{h} \eta^{2}$, $\eta \in C_{c}^{\infty}(Q(X, 2R))$ into \eqref{eqn:wf2},
and passing to the limit $h \to 0$,
we get the following energy estimate.

\begin{lemma}\label{lem:energy_estimate}
Let $f \in L^{\infty}(D \times I)$, and let $u \in V_{\loc}(D \times I)$ be
a local weak subsolution to $\mathcal{H} u = f$ in $D \times I$.
Assume that $Q(X, 2R) \subset D \times I$.
Then,
\[
\begin{split}
& 
\frac{1}{R^{n}} \iint_{Q(X, R)} |\nabla u|^{2} \, dx dt
\\
& \quad
\le
\frac{C}{R^{n + 2}} \iint_{Q(X, 2R)} |u|^{2} \, dx dt
+
C \left( R^{2} \| f \|_{L^{\infty}(Q(X, 2R))} \right)^{2},
\end{split}
\]
where $C$ is a constant depending only on $n$ and $L$.
\end{lemma}

Using De Giorgi or Moser's iteration technique, we get the following $L^{\infty}$ estimate.
See, e.g., \cite[Chapter III, Section 7]{MR241822} or \cite[Theorem 6.17]{MR1465184}.

\begin{lemma}\label{lem:global_boundedness}
Assume that $\diam (D) \le R$.
Let $u \in V(D \times (0, T))$ be a weak solution to $\mathcal{H} u = f$ in $D \times (0, T)$.
Then, we have
\[
\esssup_{ D \times (0, T)} u 
\le
\max \left\{ \sup_{0 < t < T} \sup_{\partial D} u(\cdot, t), \ \esssup_{x \in D} u(x, 0) \right\}
+
C R^{2} \| f \|_{L^{\infty}(D \times (0, T))}
\]
where $C$ is a positive constant depending only on $n$ and $L$.
\end{lemma}

For nonnegative supersolutions to $\mathcal{H} u = f$,
we have the following weak Harnck inequality (Lemma \ref{lem:whi}),
which implies local H\"{o}lder continuity of solutions,
Therefore, we treat $u$ as continuous without further explanation.
This is a classical estimate (see, e.g., \cite{MR159139, MR288405, MR1465184}), and for more recent proofs,
see, for example, \cite{MR2760150}.

\begin{lemma}\label{lem:whi}
Let $f \in L^{\infty}(D \times (0, T))$.
Assume that $\mathcal{H} u \ge f$, $u \ge 0$ in $D \times (0, T)$
and $Q( X_{0} , 4R) \subset D \times (0, T)$.
Let $0 < p_{\star} < (n + 2) / n$.
Then, there exists a constant $C$ depending only on $n$, $L$ and $p_{\star}$ such that
\[
\left( \frac{1}{R^{n + 2}} \iint_{ Q_{-} } u^{p_{\star}} \, dx dt  \right)^{1 / p_{\star}}
\le
C \left(
\essinf_{Q( X_{0}, R )} u + R^{2} \| f_{-} \|_{ L^{\infty}(X_{0}, 4R) }
\right)
\]
where $X_{0} = (x_{0}, t_{0})$, $Q_{-} = B(x_{0}, 3R) \times (t_{0} - 16 R^{2}, t_{0} - 8 R^{2})$.
\end{lemma}

Using this lemma and a Harnack chain argument, we obtain the strong minimum principle:
Let $D$ be a connected open set.
If $u$ is a nonnegative supersolution to $\mathcal{H} u \ge 0$ in $D \times (0, T)$, 
and if $u(x, T) = 0$ for some $x \in D$, then $u(y, s) \equiv 0$ for all $(y, s) \in D \times (0, T)$.

We also have the following comparison principle.

\begin{lemma}\label{lem:comparison}
Assume that $u, v \in  V_{\loc}(D \times (0, T)) \cap C(\cl{ D \times (0, T)} )$ satisfy
$\mathcal{H} u \le \mathcal{H} v$ in $D \times (0, T)$,
$u \le v$ on $\partial D \times (0, T)$ and 
$u(\cdot, 0) \le v(\cdot, 0)$ in $D$.
Then, $u \le v$ in $D \times (0, T)$.
\end{lemma}

The following interior H\"{o}lder estimate is a standard consequence of Lemma \ref{lem:whi}.

\begin{lemma}\label{lem:interior_hoelder_esti}
$Q(X_{0}, 2R) \subset D \times (0, T)$.
Assume that $\mathcal{H} u = f$ in $D \times (0, T)$ and $f \in L^{\infty}(D \times (0, T))$.
Then, $u$ is H\"{o}lder continuous at $X_{0}$.
Moreover,
\begin{equation}
\osc_{ Q(X_{0}, r) } u
\le
C \left( \frac{r}{R} \right)^{\alpha_{0}} \left( \osc_{Q(X_{0}, R)} u + R^{2} \| f \|_{L^{\infty}(Q(X_{0}, R) )} \right)
\end{equation}
for all $0 < r \le R$, where $C$ and $\alpha_{0}$ are constants depending only on
$n$, $L$.
\end{lemma}

Finally, we state a boundary H\"{o}lder estimate under the condition \eqref{eqn:CDC}, which serves as a primary tool in the subsequent section. Although the boundary regularity of parabolic equations under the capacity density condition is well-known, the formulation presented there lacks the precise quantitative form required for our argument in Lemma \ref{lem:boundary_regularity_esti}. Since this quantitative estimate is essential for the barrier construction, we provide a proof in Appendix \ref{sec:boundary_regularity},
modifying the proof in \cite{MR563383}.

\begin{lemma}\label{lem:boundary_regularity_esti}
Assume that $D \subset \R^{n}$ satisfies \eqref{eqn:CDC} for some $\gamma > 0$.
Fix $R > 0$, and let $\Xi \in \Gamma \times ((4R)^{2}, T)$.
Let $u \in V( D \times (0, T) )$ be a weak solution to
the problem $\mathcal{H} u = f$ in $D \times (0, T)$.
Then,
\[
\osc_{ Q( \Xi, r) \cap D \times (0, T) } u
\le
C_{1}
\left( \frac{r}{R} \right)^{\alpha_{B}}
\osc_{ Q( \Xi, 4R) \cap D \times (0, T) } u
+
\osc_{(\partial D) \times (0, T) \cap Q( \Xi, 4R ) } u
+
C_{2} k(4R)
\]
for all $0 < r < R$,
where $k(R) = \| f \|_{L^{\infty}(Q(\xi, R)) \cap D \times (0, T))}$
and $C_{1}$, $C_{2}$ and $\alpha_{B}$ are positive constant depending only on $n$, $L$ and $\gamma$.
\end{lemma}

\section{Construction of barrier}\label{sec:barrier}

The main result in this section is Theorem \ref{thm:main}.
We divide the proof of it into three parts (Lemmas \ref{lem:barrier1} and \ref{lem:barrier2} and the proof of Theorem \ref{thm:main}).
The necessity of \eqref{eqn:CDC} is discussed in Theorem \ref{thm:necessity}.
For simplicity, we assume that $T = \infty$ in this section.

\begin{theorem}\label{thm:main}
Assume that $D$ satisfies \eqref{eqn:CDC}.
Then, there exists $s_{\Gamma} \in V_{\loc}(D \times \R_{+} ) \cap C( D \times \R_{+} )$ such that
\begin{equation}\label{eqn:barrier1}
\mathcal{H} s_{\Gamma} \ge c_{H} \frac{ s_{\Gamma} }{ \delta_{\Gamma}^{2} } \quad \text{in} \ D \times \R_{+},
\end{equation}
\begin{equation}\label{eqn:barrier2}
\delta_{\Gamma}(X)^{\alpha_{H}}
\le
s_{\Gamma}(X)
\le
15 \delta_{\Gamma}(X)^{\alpha_{H}}
\end{equation}
for all $X \in D \times \R_{+}$.
Here, $c_{H}$ and $\alpha_{H}$ are constants depending only on $n$, $L$ and $\gamma$.
\end{theorem}

We construct the following auxiliary functions using Lemma \ref{lem:boundary_regularity_esti}.

\begin{lemma}\label{lem:barrier1}
Assume that \eqref{eqn:CDC} holds.
Let $B \subset \R^{n}$ be a ball centered at $\xi \in \Gamma$ with radius $R$.
Then, there exists a function $u_{B} \in C( D \times \R_{+} )$ such that
\begin{equation}\label{eqn:barrier1_1}
\frac{1}{4} \le u_{B} \le \frac{5}{4} \quad \text{in} \ (D \cap B) \times \R_{+},
\end{equation}
\begin{equation}\label{eqn:barrier1_2}
u_{B} = 1 \quad \text{on} \ (D \setminus B) \times \R_{+},
\end{equation}
\begin{equation}\label{eqn:barrier1_3}
\mathcal{H} u_{B}
= c R^{-2}
\quad \text{in} \ (D \cap B) \times \R_{+},
\end{equation}
\begin{equation}\label{eqn:barrier1_4}
u_{B} \le \frac{5}{16} \quad \text{on} \ \left( D \cap \theta B \right) \times \left(\frac{1}{12} R^{2}, \infty \right).
\end{equation}
Here, $\theta$ and $c$ are positive constants depending only on $n$, $L$ and $\gamma$.
\end{lemma}

\begin{proof}
Take a smooth function $\eta$ on $B = B(\xi, R)$ such that
$\eta = 1$ on $\partial B$ and $\eta = 1 / 4$ on $\cl{B} / 2$.
Let $u$ be the solution to the Cauchy-Dirichlet problem
\begin{align*}
\mathcal{H} u = c R^{-2} \quad & \text{in} \ (D \cap B) \times \R_{+},
\\
u = \eta \quad & \text{on} \ \partial_{p}( (D \cap B) \times \R_{+} ).
\end{align*}
We extend $u$ to $D \times \R_{+}$ by $u(X) = 1$ for $X \in (D \setminus B) \times \R_{+}$.

By definition, $u_{B}$ satisfies \eqref{eqn:barrier1_2} and \eqref{eqn:barrier1_3}.
By Lemma \ref{lem:global_boundedness} and the comparison principle, \eqref{eqn:barrier1_1} holds for sufficiently small $c$.
Then, by Lemma \ref{lem:boundary_regularity_esti},
\eqref{eqn:barrier1_4} is satisfied by taking $c$ even smaller and choosing $\theta$ sufficiently small as well.
By Lemma \ref{lem:boundary_regularity_esti}, $u$ is continuous at for any $X \in (D \cap \partial B) \times \R_{+}$.
Therefore, $u_{B} \in C(D \times \R_{+} )$.
\end{proof}

Using a positive constant $\theta$ in Lemma \ref{lem:barrier1},
we define a sequence of subsets of $D \times \R_{+}$.
For $k \in \Z$, we set
\[
E_{k}
:=
\left\{
X \in D \times \R_{+} \colon \delta_{\Gamma}(X) \le \left( \frac{\theta}{2} \right)^{k}
\right\}.
\]
Take a countable subset $\Gamma_{k} \subset \Gamma$ such that
\begin{equation}\label{eqn:def_xi_k}
\left\{ x \in D \colon d_{\R^{n}}(x, \Gamma) \le \left( \frac{\theta}{2} \right)^{k + 1} \right\}
\subset
\bigcup_{ \xi \in \Gamma_{k} } B\left(\xi, 2 \left( \frac{\theta}{2} \right)^{k + 1} \right).
\end{equation}
Using this $\Gamma_{k}$, we define $D_{k} \subset D$ by
\[
D_{k}
:=
D \cap
\bigcup_{ \xi \in \Gamma_{k} } B\left(\xi, \left( \frac{\theta}{2} \right)^{k} \right).
\]
Finally, we define $U_{k} \subset D \times \R_{+}$ by
\begin{equation}\label{eqn:def_u_k}
U_{k} = D_{k} \times \R_{+} \cup D \times \left( 0,  \left( \frac{\theta}{2} \right)^{2k} \right).
\end{equation}
By definition, $E_{k + 1} \subset U_{k} \subset E_{k}$.

\begin{lemma}\label{lem:barrier2}
Assume that Lemma \ref{lem:barrier1} holds.
Then, there exists $v_{k} \in V_{\loc}( D \times \R_{+} ) \cap C( D \times \R_{+} )$ such that
\begin{equation}\label{eqn:barrier2_1}
\frac{1}{4} \le v_{k} \le \frac{5}{4} \quad \text{in} \ D \times \R_{+},
\end{equation}
\begin{equation}\label{eqn:barrier2_2}
v_{k} = 1 \quad \text{on} \ (D \times \R_{+}) \setminus U_{k},
\end{equation}
\begin{equation}\label{eqn:barrier2_3}
\mathcal{H} v_{k} \ge c \left( \frac{2}{\theta} \right)^{2 k} \quad \text{in} \ U_{k},
\end{equation}
\begin{equation}\label{eqn:barrier2_4}
v_{k} \le \frac{5}{16} \quad \text{in} \ E_{k + 1},
\end{equation}
where, $c$ is a small positive constant depending only on $n$, $L$ and $\gamma$.
\end{lemma}

\begin{proof}
We define $v_{k}(x, t)$ by
\[
\tilde{v}_{k}(x, t)
:=
\inf_{ \xi \in \Gamma_{k} } v_{ B(\xi, (\theta / 2)^{k}) }(x, t)
\]
and
\[
v_{k}(x, t)
=
\min \left\{ \tilde{v}_{k}(x, t), \  \frac{1}{4} + \frac{3}{4} \left( \frac{2}{\theta} \right)^{2 k} t \right\}.
\]

Let us prove that this $v_{k}$ satisfies the desired properties.
For each $X \in D \times \R_{+}$, $v_{k}(X)$ is defined as the minimum of a finite number of $V_{\loc}( D \times \R_{+} ) \cap C( D \times \R_{+} )$-functions.
Therefore, $v_{k} \in V_{\loc}( D \times \R_{+} ) \cap C( D \times \R_{+} )$.
The inequalities \eqref{eqn:barrier2_1} and \eqref{eqn:barrier2_2} follow from \eqref{eqn:barrier1_1} and \eqref{eqn:barrier1_2} immediately.
The differential inequality \eqref{eqn:barrier2_3} follows from Lemma \ref{lem:gluing}.
In fact, $v_{k}$ is the minimum of a finite collection of supersolutions to supersolutions to \eqref{eqn:barrier2_3}. %on $U_{k}$.
Let us prove \eqref{eqn:barrier2_4}.
By \eqref{eqn:barrier1_4}, this inequality holds if
$d_{\R^{n}}(x, \Gamma) \le \left( \theta / 2 \right)^{k + 1}$
and $t > (1 / 12) \left( \theta / 2 \right)^{2 k}$.
Meanwhile, by definition, we have
\[
v_{k}(x, t)
\le
\frac{1}{4} + \frac{3}{4} \left( \frac{2}{\theta} \right)^{2 k} t
\le \frac{5}{16}
\]
for all $x\in D$ and $t \le (1 / 12) \left( \theta / 2 \right)^{2 k}$.
These results are enough to cover $E_{k + 1}$.
\end{proof}

\begin{proof}[Proof of Theorem \ref{thm:main}]
For each $k$, let $v_{k}$ be a function in Lemma \ref{lem:barrier2}.
We define a function $s$ on $D \times \R_{+}$ by
\[
s(X) := \inf_{ E_{k} \ni X } \left( \frac{1}{3} \right)^{k} v_{k}(X).
\]
Then,
\[
\left( \frac{1}{3} \right)^{k - 2} v_{k - 2}(X)
\ge
\left( \frac{1}{3} \right)^{k - 2} \frac{1}{4}
\ge
\left( \frac{1}{3} \right)^{k} \frac{5}{4}
\ge
\left( \frac{1}{3} \right)^{k} v_{k}(X)
\]
for any $X \in E_{k}$.
Therefore, we have
\[
s(X)
=
\min \left\{
\left( \frac{1}{3} \right)^{k - 1} v_{k - 1}(X), \
\left( \frac{1}{3} \right)^{k} v_{k}(X)
\right\}
\]
for all $X \in E_{k} \setminus E_{k + 1}$.
Hence, $s \in V_{\loc}(D \times \R_{+}) \cap C(D \times \R_{+})$.

We claim that
\begin{equation}\label{eqn:main1_1}
\mathcal{H} s \ge c \left( \frac{1}{3} \right)^{k} \left( \frac{2}{ \theta } \right)^{2 (k - 1)}
\end{equation}
in an open neighborhood of $E_{k} \setminus E_{k + 1}$.
By definition, $v_{k}$ and $v_{k+ 1}$ satisfy \eqref{eqn:main1_1} in $U_{k}$.
By Lemma \ref{lem:gluing}, the desired claim holds in $U_{k} \setminus E_{k + 1}$.
Since $v_{k} = 1$ in $(D \times I) \setminus U_{k}$, we have
\[
v_{k - 1} \le \frac{5}{16} < \frac{1}{3} = \frac{1}{3} v_{k} \quad \text{in} \ E_{k} \setminus U_{k}.
\]
By continuity of $v_{k}$ and $v_{k - 1}$, the same inequality holds in a neighborhood of $\cl{E_{k}} \setminus U_{k}$.
Hence, the desired claim holds.

Since $v_{k}$ satisfies \eqref{eqn:barrier2_1}, we have
\begin{equation}\label{main1_2}
\frac{1}{4} \left( \frac{1}{3} \right)^{k}
\le
s(X)
\le
\frac{15}{4} \left( \frac{1}{3} \right)^{k}
\end{equation}
for all $X \in E_{k} \setminus E_{k + 1}$.
Plugging the latter inequality and \eqref{eqn:main1_1}, we also get
\[
\mathcal{H} s
\ge
\frac{1}{4} \left( \frac{1}{3} \right)^{k} \left( \frac{2}{\theta} \right)^{2 (k - 1)}
\ge
c_{H} \frac{s}{ \delta_{\Gamma}^{2} }.
\]
Here, 
\[
c_{H} := \frac{4}{15} \left( \frac{\theta}{2} \right)^{4} c .
\]
The right-hand side is independent of $k$, and thus, this inequality holds in $D \times \R_{+}$.
Introducing $\alpha_{H} > 0$ such that $1 / 3 = (\theta / 2)^{ \alpha_{H} }$, we write \eqref{main1_2} as
\[
\frac{1}{4} \delta_{\Gamma}(X)^{ \alpha_{H} }
\le
s(X)
\le
\frac{15}{4} \delta_{\Gamma}(X)^{ \alpha_{H} }.
\]
These inequalities show that $s_{\Gamma} = 4s$ has the desired properties.
\end{proof}

Finally, we derive \eqref{eqn:CDC} from the statement of Theorem \ref{thm:main} and examine its sharpness.

\begin{theorem}\label{thm:necessity}
Assume the existence of $s_{\Gamma}$ in Theorem \ref{thm:main} with $A(x) = \mathrm{Id}_{\R^{n}}$.
Then, the condition \eqref{eqn:CDC} holds.
Furthermore, the constant $\gamma$ is estimate from below by $n$, $c_{H}$ and $\alpha_{H}$.
\end{theorem}

\begin{proof}

Let $B = B(\xi, R)$ be a ball centered at $\xi \in \Gamma$ with radius $R > 0$.
By choosing a geometric constant $c(n)$ sufficiently small, one can ensure the existence of a function $v$ satisfying
\begin{align*}
\laplacian v = c(n) R^{-2} \mathbf{1}_{B \setminus \cl{B} / 2}(x) \quad & \text{in} \ 2B,
\\
v \ge 1 \quad & \text{in} \ 2B \setminus \cl{B},
\\
v = 0 \quad & \text{on} \ \cl{B} / 2.
\end{align*}
%Let $u$ be the solution to the problem
Using this constant $c(n)$, we consider the weak solution $u$ to the problem
\begin{align*}
(\partial_{t} - \laplacian) u = c(n) R^{-2} \mathbf{1}_{B \setminus \cl{B} / 2}(x) \quad & \text{in} \ (D \cap 2B) \times \R_{+},
\\
u = 0 \quad & \text{on} \ \partial (D \cap 2B) \times \R_{+},
\\
u(\cdot, 0) = 0 \quad & \text{in} \ D.
\end{align*}
Since $D \cap 2B$ is bounded, the Poincar\'{e} and Gronwall inequalities implies that
$u(\cdot, t)$ converges to $u_{\infty}$ in $L^{2}(D \cap 2B)$ as $t \to \infty$, 
where $u_{\infty}$ is the weak solution to the stationary problem
\begin{align*}
- \laplacian u_{\infty} = c(n) R^{-2} \mathbf{1}_{B \setminus \cl{B} / 2}(x) \quad & \text{in} \ D \cap 2B,
\\
u_{\infty} = 0 \quad & \text{on} \ \partial (D \cap 2B).
\end{align*}
By assumption, $s_{\Gamma} \ge 0$ and
\[
(\partial_{t} - \laplacian) s_{\Gamma}
\ge
c_{H} \frac{s_{\Gamma}}{\delta_{\Gamma}^{2}}
\ge
c_{H} \delta_{\Gamma}^{\alpha_{H} - 2}
\ge
c_{H} (2R)^{\alpha_{H} - 2}
\]
in $(D \cap 2B) \times \R_{+}$.
Therefore, the comparison principle yields
\begin{equation*}\label{eqn:necessity1}
u(X) \le \frac{ 2^{2 - \alpha_{H} }}{c(n) c_{H} R^{\alpha_{H}}} s_{\Gamma}(X)
\le
\frac{15 \cdot 2^{2 - \alpha_{H} }}{c(n) c_{H} R^{\alpha_{H}}} \delta_{\Gamma}(X)^{\alpha_{H}}
\end{equation*}
for all $X \in (D \cap 2B) \times \R_{+}$.
Passing to the limit $t \to \infty$, we obtain
\[
u_{\infty}(x)
\le
\left(
\frac{15 \cdot 2^{2 - \alpha_{H}} }{c(n) c_{H}}
\right)
\left( \frac{ d_{\R^{n}}(x, \Gamma) }{R} \right)^{\alpha_{H}}.
\]
Next, let $\varphi$ be the solution to the minimizing problem
\[
\inf
\left\{
\int_{\R^{n}} |\nabla \varphi|^{2} \, dx
\colon
\varphi \in H_{0}^{1}(2B), \ \varphi \ge 1 \quad \text{on} \ \cl{B} \setminus D
\right\}.
\]
Defining $w := u_{\infty} + v$, we see that $w$ satisfies
\begin{align*}
- \laplacian w = 0 \quad & \text{in} \ D \cap 2B,
\\
w \ge 1 \quad & \text{on} \ (D \cap \partial 2B) \cup (\partial D \cap (2B \setminus \cl{B})).
\end{align*}
The comparison principle then implies
\[
1 - \varphi(x)
\le
w(x)
\le
\left(
\frac{15 \cdot 2^{2 - \alpha_{H}} }{c(n) c_{H}}
\right)
\left( \frac{ d_{\R^{n}}(x, \Gamma) }{R} \right)^{\alpha_{H}}
\]
for all $x \in  D \cap 2B$.
The right-hand side is less than $1 / 2$ provided that
if 
\[
\left( \frac{ d_{\R^{n}}(x, \Gamma) }{R} \right)
\le
\lambda
:=
\left(
\frac{c(n) c_{H}}{ 30 \cdot 2^{2 - \alpha_{H}} }
\right)^{1 / \alpha_{H}},
\]
which implies that $2 \varphi \ge 1$ on $\lambda B$.
Finally, by a direct calculation of the capacity of annuli, we obtain
\begin{align*}
4 \capacity( \cl{B} \setminus D, 2B )
& =
\int_{\R^{n}} |\nabla (2 \varphi) |^{2} \, dx
\\
& \ge
\capacity( \lambda \cl{B}, 2B )
\ge
c(n, \lambda) \capacity( \cl{B}, 2B ),
\end{align*}
where $c(n, \lambda)$ is a small constant depending only on $n$ and $\lambda$.
Thus, the desired inequality \eqref{eqn:CDC} holds.
\end{proof}

\section{Existence result}\label{sec:existence}

In this section, we prove an existence theorem of weak solution to \eqref{eqn:parabolic}-\eqref{eqn:IV}
using Theorem \ref{thm:main} and the comparison principle.

Throughout below, we assume that $D$ satisfies \eqref{eqn:CDC} with $\Gamma = \partial D$ and
that the \textit{inradius} of $D$ is finite:
\begin{equation}\label{eqn:inrad_cond}
\inrad{D}
:=
\| d_{\R^{n}}(\cdot, \partial D) \|_{L^{\infty}(D)}
<
\infty.
\end{equation}

\begin{example}
(1) If $D$ is bounded, then \eqref{eqn:inrad_cond} holds clearly.
(2) For $0< r < \sqrt{n} / 2$, $D = \R^{n} \setminus \bigcup_{\xi \in \Z^{n}} \cl{B}(\xi, r)$ is an unbounded CDC domain satisfying \eqref{eqn:inrad_cond}.
\end{example}

For notational convenience, $\delta_{\Gamma}$ and $s_{\Gamma}$ are abbreviated as $\delta$ and $s$ below.

\begin{lemma}\label{lem:transform}
Let $\theta \colon (0, \infty) \to (0, \infty)$ be a continuously differentiable non-decreasing concave function such that
\[
\Theta( \sigma ) := \int_{0}^{ \sigma } \theta( \tau) \frac{d \tau }{ \tau } < \infty
\]
for some $\sigma > 0$.
Let $s (= s_{\Gamma})$ be a function in Theorem \ref{thm:main}.
Then, $\Theta(s)$ satisfies
\[
\mathcal{H} \Theta(s) \ge c_{H} \frac{ \theta( \delta^{ \alpha_{H} } )}{ \delta^{2} } \quad \text{in} \ D \times (0, T),
\]
where $c_{H}$ and $\alpha_{H}$ are the constants in Theorem \ref{thm:main}.
\end{lemma}

\begin{example}
(i) $\theta(t) = t^{\beta}$ ($0 < \beta \le 1$) is concave and satisfies the above integral condition.
In this case, $\Theta(t) = \beta^{-1} t^{\beta}$.
(ii) $\theta(t) = (\log t)^{- \beta}$ ($\beta > 1$) is concave on $[0, e^{- (\beta + 1)}]$
and satisfies the above integral condition.
\end{example}

\begin{proof}
By assumption, $\theta'(t) t \le \theta(t)$ for all $t \in (0, \infty)$, and thus $\Theta$ is increasing and concave on $(0, \infty)$.
Fix a nonnegative function $\eta \in C^{\infty}_{c}(D \times (0, T))$.
Substituting the test function $\varphi = \Theta'( s_{h} ) \eta$ into \eqref{eqn:wf2},
we get
\begin{align*}
& 
\iint_{D \times (0, T)}
\frac{ \partial s_{h} }{\partial t} ( \Theta'( s_{h} ) \eta ) + [ A \nabla s ]_{h} \cdot \nabla ( \Theta'( s_{h} ) \eta )
\, dx dt
\\
& \ge
c_{H} 
\iint_{D \times (0, T)}
\left[ \frac{ s }{ \delta^{2} } \right]_{h} \left( \Theta'( s_{h} ) \eta \right) \, dx dt.
\end{align*}
Passing to the limit $h \to 0$, we get
\begin{equation*}
\begin{split}
&
- \iint_{ D \times (0, T) } \Theta( s ) \frac{ \partial \eta }{ \partial t } \, dx dt
+
\iint_{ D \times (0, T) } A \nabla s  \cdot \nabla ( \Theta'( s ) \eta ) \, dx dt
\\
& \ge
c_{H} 
\iint_{ D \times (0, T) }
\frac{ s }{ \delta^{2} } \left( \Theta'( s ) \eta \right) \, dx dt.
\end{split}
\end{equation*}
Consider the third term on left-hand side.
Since $\Theta'' \le 0$, it follows from \eqref{eqn:elliptic} that
\[
\iint_{ D \times (0, T) } A \nabla \Theta'( s ) \cdot \nabla \eta \, dx dt
\ge
\iint_{ D \times (0, T) } A \nabla s  \cdot \nabla ( \Theta'( s ) \eta ) \, dx dt.
\]
On the other hand, the integrand on the right-hand is estimated from below by
\[
s \, \Theta'( s ) = \theta( s ) \ge \theta( \delta^{\alpha} ).
\]
Combining them, we obtain the desired differential inequality.
\end{proof}

\begin{theorem}\label{thm:existence}
Assume that $D$ satisfies \eqref{eqn:CDC} and \eqref{eqn:inrad_cond}.
Let $f \in L^{1}_{\loc}(D \times (0, T))$, and assume that there exists a function $\theta$ satisfying the assumption in Lemma \ref{lem:transform} and
\begin{equation}\label{eqn:force_term}
|f(X)|
\le c_{H} \frac{ \theta( \delta(X)^{\alpha_{H}}) }{ \delta(X)^{2} } 
\end{equation}
for almost every $X \in D \times (0, T)$, where $c_{H}$ and $\alpha_{H}$ are the constants in Theorem \ref{thm:main}.
Then, there exists a unique weak solution $u \in V_{\loc}(D \times (0, T) ) \cap C( \cl{ D \times (0, T) } )$ to \eqref{eqn:parabolic}-\eqref{eqn:IV} 
in the sense that
\begin{equation}\label{eqn:boundary_esti_for_u}
|u(X)|
\le \Theta( 15 \delta(X)^{\alpha_{H}} )
\end{equation}
holds for all $X \in D \times (0, T)$.
\end{theorem}

\begin{proof}
Without loss of generality, we may assume that $f \ge 0$.
Take a sequence $\{ f_{k} \}_{k = 1}^{\infty} \subset L^{\infty}(D \times (0, T))$
such that $f_{k} \uparrow f$ in $D \times (0, T)$.
We consider the approximated problems
\begin{align*}
\mathcal{H} u_{k} = f_{k} \quad & \text{in} \ D \times \R_{+},
\\
u_{k} = 0 \quad & \text{on} \ \partial D \times \R_{+},
\\
u_{k}(\cdot, 0) = 0 \quad & \text{in} \ D.
\end{align*}
By interior regularity theory, we may assume that $u_{k} \in C(D \times (0, T])$.
By the comparison principle and Lemma \ref{lem:transform}, we have
\begin{equation}\label{eqn:barrier_bound_k}
0 \le u_{k}(X) \le \Theta(s(X)) \le \Theta( 15 \delta(X)^{\alpha} )
\end{equation}
for all $X$. 
Consider the bounded function $u(X) = \lim_{k \to \infty} u_{k}(X)$.
By Lemma \ref{lem:energy_estimate}, 
\begin{align*}
\nabla u_{k} \wkto \nabla u \quad & \text{weakly in } L^{2}_{\loc}(D \times (0,T)).
\end{align*}
Thus, $\mathcal{H} u = f$ in $D \times (0, T)$.
By interior regularity, $u \in C(D \times (0, T])$.
Since $u$ is controlled by $\delta(X)$, it also belongs to $C( \cl{D \times (0, T)} )$.
\end{proof}

Let $0 < \alpha \le 1$ and $E \subset \R^{n} \times \R$.
For a function $u$ on $E$, we define the parabolic H\"{o}lder seminorm as
\[
[ u ]_{ \alpha, E \times I}
:=
\sup_{X, Y \in E}
\frac{ |u(X) - u(Y)| }{ d_{p}(X, Y)^{\alpha} }.
\]
For $E = D \times I$,
the parabolic H\"{o}lder space $C^{\alpha, \alpha / 2}(D \times I)$ is then defined by
\[
\| u \|_{C^{\alpha, \alpha / 2}( D \times I )}
:=
\| u \|_{L^{\infty}( D \times I )}
+
(\inrad D)^{\alpha} [ u ]_{ \alpha, D \times I }.
\]
We prove the following.

\begin{theorem}\label{thm:main2}
Assume that \eqref{eqn:CDC} holds with $\Gamma = \partial D$ and $\inrad D < \infty$.
Let $0 < \alpha \le 1$.
Assume that $f \delta^{2 - \alpha} \in L^{\infty}(D \times (0, T))$.
Then, there exists a unique weak solution to the problem \eqref{eqn:parabolic}-\eqref{eqn:IV}.
Moreover, 
\begin{align*}
\| u \|_{C^{ \alpha_{\star}, \alpha_{\star} / 2 }( D \times (0, T)) }
\le
C \left(
(\inrad D)^{\alpha }\| f \delta^{2 - \alpha} \|_{L^{\infty}(D \times (0, T))}
\right),
\end{align*}
where $C_{\star}$ and $\alpha_{\star}$ are constants depending only on
$n$, $L$, $\gamma$ and $\alpha$.
\end{theorem}

\begin{proof}
The existence of this solution in $V_{\loc}(D \times (0, T)) \cap C( \cl{D \times (0, T)} )$ follows from Theorem \ref{thm:existence} directly.
As in \eqref{eqn:boundary_esti_for_u}, $u(X)$ is controlled by $\delta(X)^{\alpha_{H}}$.
Thus, it also belongs to $C(\cl{D})$.
In particular, for $\alpha_{1} = \min\{ \alpha_{H}, \alpha \}$, we have
\begin{equation}\label{eqn:boundary_esti_for_u-2}
|u(X)| \le C M (\inrad D)^{ - \alpha_{1} } \delta(X)^{ \alpha_{1} },
\end{equation}
where $M = (\inrad D)^{\alpha }\| f \delta^{2 - \alpha} \|_{L^{\infty}(D \times (0, T))}$.

We set $\alpha_{\star} = \min\{ \alpha_{0}, \alpha_{1} \}$, where $\alpha_{0}$ is a positive exponent in Lemma \ref{lem:interior_hoelder_esti}.
Take any $X, Y \in D \times (0, T)$.
If $d_{p}(X, Y) > \max\{ \delta(X),  \delta(Y) \} / 2$, then, \eqref{eqn:boundary_esti_for_u-2} gives
\[
\begin{split}
|u(X) - u(Y)|
& \le
|u(X)| + |u(Y)|
\\
& \le
\left( C M \right)
R^{- \alpha_{0}} \left( \delta(X) + \delta(Y) \right)^{ \alpha_{0} }
\\
& \le
\left( C M \right)
R^{- \alpha_{\star}} d_{p}(X, Y)^{ \alpha_{\star} }.
\end{split}
\]
Therefore, we may assume that $d_{p} (X, Y) \le \delta(X) / 2$ without loss of generality.
Set $r := \delta(X)$. Since $Q(X, 2 r) \subset D \times (0, T)$,  Lemma \ref{lem:interior_hoelder_esti} gives
\[
\begin{split}
|u(X) - u(Y)|
& \le
C \left( \frac{ d_{p}(X, Y) }{r} \right)^{ \alpha_{\star} }
\left( \osc_{Q(X, r )} u + r^{2} \| f \|_{L^{\infty}(Q(X, r))} \right)
\\
& \le
C \left( \frac{ d_{p}(X, Y) }{r} \right)^{ \alpha_{\star} }
\left( \osc_{Q(X, r )} u + r^{\alpha} M \right).
\end{split}
\]
Meanwhile, by \eqref{eqn:boundary_esti_for_u-2}, we have
\[
\begin{split}
|u(Z)|
\le
\left( C M \right) R^{ - \alpha_{1} } r^{ \alpha_{1} }
\le
\left( C M \right) R^{ - \alpha_{\star} } r^{ \alpha_{\star} }
\end{split}
\]
for all $Z \in Q(x, r)$.
Combining the two inequalities, we obtain the desired H\"{o}lder seminorm estimate.
\end{proof}

\begin{corollary}
Assume that $D$ is bounded and \eqref{eqn:CDC} holds with $\Gamma = \partial D$.
Let $0 < \alpha \le 1$.
Assume also that $f \delta^{t} \in L^{\infty}(0, T; L^{q}(D))$
for some $n / 2 < q < \infty$ and $0 \le t < 2 - n / q$.
Then, there exists a unique weak solution to the problem \eqref{eqn:parabolic}-\eqref{eqn:IV}.
Moreover, 
\begin{align*}
\| u \|_{C^{ \alpha_{\star}, \alpha_{\star} / 2 }( D \times (0, T)) }
\le
C (\diam D)^{ 2 - n / q - t }\| f \delta^{t}  \|_{L^{\infty}(0, T; L^{q}(D)) },
\end{align*}
where $C_{\star}$ and $\alpha_{\star}$ are constants depending only on
$n$, $L$, $q$, $t$ and $\gamma$.
\end{corollary}

\begin{proof}
Assume first that $f \in L^{\infty}(0, T; L^{q}(D))$ with $n / 2 < q < \infty$.
Then, the results in \cite{MR241822} are extended by the arguments mentioned above and Section \ref{sec:boundary_regularity} to yield
\begin{align*}
\| u \|_{C^{ \alpha_{\star}, \alpha_{\star} / 2 }( D \times (0, T)) }
\le
C (\diam D)^{ 2 - n / q }\| f  \|_{L^{\infty}(0, T; L^{q}(D)) },
\end{align*}
where $C$ and $\alpha_{\star}$ is a positive constant depending only on $n$, $L$, $q$ and $\gamma$.
The desired result follows from the interpolation theorem.
\end{proof}

\section{Inhomogeneous boundary value}\label{sec:boundary_value_problem}

Finally, we extend Theorems \ref{thm:main2} to the following inhomogeneous boundary value problem
\begin{align}
\mathcal{H} u = f \quad & \text{in} \ D \times (0, T),
\label{eqn:parabolic2}
\\
u = g \quad & \text{on} \ \partial D \times (0, T),
\label{eqn:BV2}
\\
u(\cdot, 0) = u_{0} \quad & \text{in} \ D.
\label{eqn:IV2}
\end{align}
For the problem to be well-posed and to ensure the continuity of the solution up to the parabolic boundary, we assume the $0$-th order compatibility condition
\begin{equation}\label{eqn:compatibility_cond}
g(\xi, 0) = u_0(\xi) \quad \text{for all } \xi \in \partial D.
\end{equation}
Under this assumption, we introduce the norms
\[
\| g \|_{C^{\alpha, \alpha / 2}( \partial D \times (0, T) )}
:=
\| g \|_{L^{\infty}( \partial D \times (0, T) )}
+
(\inrad D)^{\alpha} [ g ]_{ \alpha, \partial D \times (0, T) }
\]
and
\[
\| u_{0} \|_{C^{\alpha}( D )}
:=
\| u_{0} \|_{L^{\infty}( D )}
+
(\inrad D)^{\alpha} [ u_{0} ]_{ \alpha, D }.
\]

\begin{theorem}\label{thm:main3}
Assume that \eqref{eqn:CDC} holds with $\Gamma = \partial D$ and $\inrad D < \infty$.
Let $0 < \alpha \le 1$, and assume that $f \delta^{2 - \alpha} \in L^{\infty}(D \times (0, T))$.
Assume that $g \in C^{\alpha}(\partial D \times (0, T))$ and $u_{0} \in C^{\alpha}(\cl{D})$ satisfy \eqref{eqn:compatibility_cond}.
Then, there exists a unique weak solution $u \in V_{\loc}(D \times (0, T)) \cap C(\cl{D \times (0, T)})$
to the problem \eqref{eqn:parabolic2}-\eqref{eqn:BV2}.
Moreover, 
\begin{align*}
& 
\| u \|_{C^{ \alpha_{\star}, \alpha_{\star} / 2 }( D \times (0, T)) }
\\
& \quad
\le
C \left(
(\inrad D)^{\alpha }\| f \delta^{2 - \alpha} \|_{L^{\infty}(D \times (0, T))}
+
\| g \|_{C^{\alpha. \alpha / 2}( \partial D \times (0, T))}
+
\| u_{0} \|_{C^{\alpha}(D)}
\right),
\end{align*}
where $C_{\star}$ and $\alpha_{\star}$ are constants depending only on
$n$, $L$, $\gamma$ and $\alpha$.
\end{theorem}

It suffices to show the following statement. We establish this using a combination of Duhamel's formula around the initial time and standard parabolic regularity.

\begin{theorem}\label{thm:main4}
Assume that \eqref{eqn:CDC} holds with $\Gamma = \partial D$ and $\inrad D < \infty$.
Assume that $g \in C^{\alpha}(\partial D \times (0, T))$ and $u_{0} \in C^{\alpha}(\cl{D})$ satisfy \eqref{eqn:compatibility_cond}.
Let $u$ be a solution to \eqref{eqn:parabolic2}-\eqref{eqn:IV2} with respect to $f = 0$.
Then,
\[
\| u \|_{C^{\alpha_{\star \star}}(D \times (0, T))}
\le
C \left( \| g \|_{C^{\alpha, \alpha / 2}( \partial D \times (0, T)) } + \| u_{0} \|_{C^{\alpha}(D)} \right),
\] 
where $C$ and $\alpha_{\star \star}$ are constants depending only on $n$, $L$, $\gamma$ and $\alpha$.
\end{theorem}

\begin{proof}
\noindent \textbf{Step 1.}
Set $M := \| g \|_{C^{\alpha, \alpha / 2}( \partial D \times (0, T)) ) } + \| u_{0} \|_{C^{\alpha}(D)}$.
Using the McShane extension, we define a H\"{o}lder continuous function $u$ on $\R^{n}$ by
\begin{equation}\label{eqn:mcshane}
\tilde{ u_{0} }(x)
:=
\max\left\{ \sup_{a \in \cl{D} }( u(a) -  [ u_{0} ]_{\alpha, D} |x - a|^{\alpha} ), \, - \| u_{0} \|_{L^{\infty}(D)} \right\}.
\end{equation}
We also extend the parabolic operator $\mathcal{H}$ to $\R^{n} \times \R_{+}$ and denote it by $\tilde{ \mathcal{H} }$.
Let $p(x, y, t)$ be the heat kernel associate with $\tilde{ \mathcal{H} }$.
By the Aronson estimate (see, \cite{MR217444}), we have
\begin{equation}\label{eqn:aronson}
\frac{1}{C t^{n / 2}} \exp \left( - C \frac{|x - y|^{2}}{ t } \right)
\le
p(x, y, t)
\le
\frac{C}{ t^{n / 2}} \exp \left( - \frac{|x - y|^{2}}{C t} \right),
\end{equation}
for all $x, y \in \R^{n}$, $t > 0$.
Define
\[
P \tilde{u}_{0}(x, t)
=
\int_{\R^{n}} p(x, y, t) \tilde{u_{0}}(y) \, dy.
\]
Since $\int_{\R^{n}} p(x, y, t) \, dy = 1$, we have
\begin{equation}\label{eqn:boundedness_of_tilde_u}
\| P \tilde {u}_{0} \|_{L^{\infty}(\R^{n} \times \R_{+})} \le \| \tilde{u}_{0} \|_{L^{\infty}(\R^{n})}.
\end{equation}
Furthermore, since
\[
| P \tilde{u}_{0}(x, t) - u_{0}(x) |
\le
\int_{\R^{n}} p(x, y, t) | \tilde{u}_{0}(y) - u_{0}(x) | \, dy,
\]
using the latter inequality in \eqref{eqn:aronson}, we obtain
\begin{align*}
| P \tilde{u}_{0}(x, t) - u_{0}(x) |
& \le
\int_{\R^{n}} \frac{C}{ t^{n / 2}} \exp \left( - \frac{|x - y|^{2}}{C t} \right) [ \tilde{u}_{0} ]_{\alpha, \R^{n}} (\inrad{D})^{-\alpha} |x - y|^{\alpha} \, dy
\\
& \le
C [ u_{0} ]_{\alpha, D} (\inrad{D})^{-\alpha} t^{\alpha / 2}.
\end{align*}
Hence, if $d_{\R^{n}}(x, y) \ge t^{1 / 2}$ and $s \le t$, then 
\begin{align*}
| P \tilde{u}_{0}(x, t) - P u_{0}(y, s) |
& \le
C [u_{0}]_{\alpha, D} (\inrad{D})^{-\alpha} d_{p}( (x, t), (y, s))^{\alpha / 2}.
\end{align*}
Combining this inequality, Lemma \ref{lem:interior_hoelder_esti} and \eqref{eqn:boundedness_of_tilde_u}, we obtain
\[
\| P u_{0} \|_{C^{\alpha'}(D \times \R_{+})} 
\le
C \| \tilde{u}_{0} \|_{C^{\alpha}( \R^{n} )}
\le
C \| u_{0} \|_{C^{\alpha}(D)}
\le
C M,
\]
where $\alpha' := \min\{ \alpha_{0}, \alpha \}$ is a positive constant depending only on $n$, $L$ and $\alpha$. 

\textbf{Step 2.}
Let $v$ be the solution to the Dirichlet problem
\begin{align*}
\mathcal{H} v = 0 \quad & \text{in} \ D \times (0, T),
\\
v = \tilde{g} :=  g  - P u_{0} \quad & \text{on} \ \partial D \times (0, T),
\\
v(\cdot, 0) = 0 \quad & \text{in} \ D.
\end{align*}
By \eqref{eqn:compatibility_cond}, 
$\tilde{g}(\xi, 0) = 0$ for all $\xi \in \partial D$ and $\tilde{g} \in C^{\alpha', \alpha' / 2}( \partial_{p}( D \times (0, T)) )$.
By the comparison principle, we have
\begin{equation}\label{eqn:esti_of_v}
\| v(\cdot, t) \|_{L^{\infty}(D)} \le C M (\inrad D)^{-\alpha'} t^{ \alpha' / 2}.
\end{equation}
Let $X = (x, t), Y = (y, s) \in D \times (0, T)$.
It is enough to show that
\begin{equation}\label{eqn:hoelder_esti_of_v}
| v(X) - v(Y) | \le C M (\inrad D)^{- \alpha_{\star \star} } d_{p}(X, Y)^{ \alpha_{\star \star} }
\end{equation}
for all $X, Y \in D \times (0, T)$, where $\alpha_{\star \star}$ is a positive constant.

\textbf{Step 3.}
We divide the proof into two cases.
Assume first that
\begin{equation}\label{eqn:case_a}
\max\{ \delta(X), \delta(Y) \} \le 4 d_{p}(X, Y).
\end{equation}
Let  $X_{*}, Y_{*} \in \partial_{p}(D \times (0, T))$ be such that $\delta(X) = d_{p}(X, X_{*})$ and $\delta(Y) = d_{p}(Y, Y_{*})$.
Without loss of generality, we may assume that $X_{*} = (x_{*}, t)$ and $x_{*} \in \partial D$, or $X_{*} = (x, 0)$.
By the triangle inequality, we have
\[
| v(X) - v(Y) |
\le
| v(X) - v(X_{*}) | + | v(X_{*}) - v(Y_{*}) | + | v(Y_{*}) - v(Y) |.
\]
Since
\[
d_{p}(X_{*}, Y_{*})
\le
d_{p}( X_{*}, X) + d_{p}(X, Y) + d_{p}(Y, Y_{*}),
\]
the second term is estimated by
\[
\begin{split}
| v(X_{*}) - v(Y_{*}) | 
\le
[ \tilde{g} ]_{ \alpha', \partial D \times (0, T) } (d_{p}(X, Y) + 2 r)^{\alpha'}
\le
C M (\inrad D)^{- \alpha'} d_{p}(X, Y)^{\alpha'}.
\end{split}
\]
Next, we estimate the other two terms.
If $X_{*} = (x, 0)$, then \eqref{eqn:esti_of_v} yields
\[
|v(X) - v(X_{}*)|
=
| v(x, t) | \le C M (\inrad D)^{-\alpha'} t^{ \alpha' / 2} = C M (\inrad D)^{-\alpha'} \delta(X)^{\alpha'}.
\]
Let $X_{*} = (x_{*}, t)$.
We use Lemma \ref{lem:boundary_regularity_esti} near $X_{*}$.
By a well-known iteration technique (see, e.g., \cite[Lemma 4.6]{MR1465184}), for any $r \le R \le t^{1 / 2}$, we have
\begin{equation}\label{eqn:boundary_regularity}
\osc_{ Q(X_{0}, r) \cap D \times (0, T) } v
\le
C \left( \left( \frac{r}{R} \right)^{ \beta } \osc_{ Q(X_{0}, R) \cap D \times (0, T) } v
+ [ \tilde{g} ]_{\alpha', \partial D \times (0, T) } ( \sqrt{r R} )^{\alpha'}
\right).
\end{equation}
Here, $0 < \beta \le \alpha' / 2$ is a positive constant depending only on $C_{1}$, $C_{2}$ and $\alpha_{0}$.
To estimate the right-hand side, we consider two subcases.
If $t^{1 / 2} \ge \inrad D$, we choose $R = \inrad D$.
Since 
\[
\osc_{ Q(X_{0}, R) \cap D \times (0, T) } v
\le
2 \| u \|_{L^{\infty}(D \times (0, T))} \le 2M,
\]
we get
\[
\osc_{ Q(X_{0}, r) \cap D \times (0, T) } v
\le
C M (\inrad D)^{- \beta} r^{ \beta }.
\]
If $t^{1/ 2} \le \inrad D$, we set $R = t^{1 / 2}$.
Applying \eqref{eqn:esti_of_v} to the right-hand side of \eqref{eqn:boundary_regularity}, we obtain
\[
\begin{split}
\osc_{ Q(X_{0}, r) \cap D \times (0, T) } v
& \le
C M (\inrad D)^{- \alpha'} t^{\alpha' / 2 - \beta / 2} r^{\beta} + C M (\inrad D)^{- \alpha'} (r t^{1 / 2})^{\alpha' / 2}
\\
& \le
C M (\inrad D)^{- \beta} r^{\beta}.
\end{split}
\]
This shows that \eqref{eqn:hoelder_esti_of_v} holds for any $0 < \alpha_{\star \star} \le \beta$.

\textbf{Step 4.}
Suppose now that \eqref{eqn:case_a} does not hold.
Without loss of generality, we may assume $d_{p} (X, Y) \le \delta(X) / 4$.
Let $r := \delta(X)$. 
Take a cylinder $Q(Z, r / 4)$ centered at $X$ or $Y$ such that $X, Y \in Q(Z, r / 4)$.
By the triangle inequality,
\[
\frac{3}{4} r  \le \delta(Z) \le \frac{5}{4} r.
\]
Since $Q(Z, 3r / 4) \subset D \times (0, T)$, Lemma \ref{lem:interior_hoelder_esti} yields
\[
\begin{split}
|v(X) - v(Y)|
\le
C \left( \frac{ d_{p}(X, Y) }{r} \right)^{ \alpha_{0} }
\osc_{Q(Z, r / 2 )} u.
\end{split}
\]
As Step 3, we take $Z_{*}$.
Then, following the same argument as in Step 3, we obtain
\[
\begin{split}
|v(Z) - v(W)|
& \le
|v(Z) - v(Z_{*})|
+
|v(Z_{*}) - v(W)|
\\
& \le
\left( C M \right) ( \inrad{D} )^{- \beta} r^{ \beta }
\end{split}
\]
for all $W \in Q(Z, r / 2)$.
Setting $\alpha_{\star \star} = \min\{ \beta, \alpha_{0} \}$
and combining the above inequalities, we obtain the desired estimate \eqref{eqn:hoelder_esti_of_v}.
\end{proof}

\appendix

\section{Proof of Lemma \ref{lem:boundary_regularity_esti}}\label{sec:boundary_regularity}

\begin{proof}[Proof of Lemma \ref{lem:boundary_regularity_esti}]
\textbf{Step 1.}
We assume that $u$ is a nonnegative supersolution to $\mathcal{H} u = f$ in $Q(\Xi_{0}, 4R)$ in this step.
We write $\Xi_{0} = (\xi_{0}, t_{0})$ and define the following subsets of $Q(\Xi_{0}, 4R)$:
\[
Q_{1}
:=
B(\xi_{0}, R) \times (t_{0} - 12 R^{2}, t_{0} - 8 R^{2}),
\]
\[
Q_{2}
:=
B(\xi_{0}, 2R) \times (t_{0} - 16 R^{2}, t_{0} - 8 R_{2}),
\]
\[
Q_{3}
:=
B(\xi_{0}, 3R) \times (t_{0} - 16 R^{2}, t_{0} - 4 R_{2}).
\]
Following \eqref{eqn:sobolev_bv}, let $m := \inf_{\partial D \times (0, T) \cap Q_{1}} u$.
We define the truncated function $u_{m}(X)$ as
\[
u_{m}^{-}(X)
:=
\begin{cases}
\min\{ u(X), m \} 
& \text{if} \ X \in D \times (0, T),
\\
m
& \text{otherwise.}
\end{cases}
\]
Our goal is to prove that
\begin{equation}\label{eqn:boundary_esti1}
m
\le
\frac{C}{\gamma}
\left(
\essinf_{Q(\Xi_{0}, R)} u_{m}^{-} + k(4R)
\right).
\end{equation}
Here $C$ is a constant depending only on $n$ and $L$.

Take $\eta \in C_{c}^{\infty}(B(\xi, 2R))$ such that $\eta = 1$ on $\cl{B(\xi, R)}$ and $|\nabla \eta| \le C / R$.
Then, $u_{m}^{-}(\cdot, t) \eta \in H_{0}^{1}(B(\xi, 2R))$ and $u_{m}^{-} \eta \ge 1$ a.e. on $\cl{B(\xi, R)}$
for all $t \in (t_{0} - 12 R^{2}, t_{0} - 8 R^{2})$.
By \eqref{eqn:CDC}, we have
\[
m^{2} \gamma R^{n - 2}
\le
C \int_{B(\xi, 2R)} |\nabla (u_{m}^{-}(\cdot, t) \eta )|^{2} \, dx
\]
for all $t \in (t_{0} - 12 R^{2}, t_{0} - 8 R^{2})$.
Integrating this over the specified time interval, we obtain
\begin{equation}\label{eqn:boundary_esti2}
\begin{split}
m^{2} \gamma R^{n}
& \le
C \iint_{Q_{1}} |\nabla (u_{m}^{-}(\cdot, t) \eta )|^{2} \, dx dt
\\
& \le
\frac{2C}{R^{2}}
\iint_{Q_{1}} |u_{m}^{-}(\cdot, t)|^{2} \, dx dt
+
2C \iint_{Q_{1}} |\nabla (u_{m}^{-}(\cdot, t)|^{2} \, dx dt.
\end{split}
\end{equation}
Since $u_{m}^{-} \le m$, Lemma \ref{lem:whi} yields
\[
\begin{split}
\iint_{Q_{2}}
(u_{m}^{-})^{2}
\, dx dt
& \le
m \iint_{Q_{2}}
u_{m}^{-}
\, dx dt
\\
& \le
m
\cdot
C \left(
\essinf_{Q( \Xi_{0}, R)} u_{m}^{-} + k(4R)
\right)
R^{n + 2}.
\end{split}
\]
Thus, the former term on the right-hand side of \eqref{eqn:boundary_esti2}
is controlled by the right-hand side of \eqref{eqn:boundary_esti1}.
Let us consider the latter term.
Take a nonnegative function $\eta \in C^{\infty}(Q_{2})$ such that $\eta = 1$ on $Q_{1}$, 
$\eta = 0$ near the lateral boundary and bottom of $Q_{2}$, and $|\nabla \eta|^{2} + | \frac{ \partial \eta}{\partial t} | \le C R^{2}$.
Testing \eqref{eqn:wf2} with the nonnegative function $\varphi = [(m - u_{m}^{-})]_{h} \eta^{2}$,
we get
\[
\iint_{Q_{1}}
|\nabla u_{m}^{-}|^{2}
\, dx dt
\le
C m \left( 
\frac{1}{R}
\iint_{Q_{2}} |\nabla u_{m}^{-}| \, dx dt 
+
k(4R) R^{n}
\right).
\]
Set $v = u_{m}^{-} + k(4R)$ and fix $-2 / n < p < 0$. By H\"{o}lder's inequality, we have
\[
\iint_{Q_{2}} |\nabla v| \, dx dt
\le
\left(
\iint_{Q_{2}} |\nabla v|^{2} v^{p - 1} \, dx dt 
\right)^{1 / 2}
\left(
\iint_{Q_{2}} v^{1 - p} \, dx dt 
\right)^{1 / 2}.
\]
Take $\eta \in C^{\infty}(Q_{3})$ such that $\eta = 1$ on $Q_{2}$,
$\eta = 0$ near the lateral boundary and top of $Q_{3}$,
and $|\nabla \eta|^{2} + | \frac{ \partial \eta}{\partial t} | \le C R^{2}$.
Since $v$ satisfies $\mathcal{H} v \ge f$ in $Q(4R)$,
substituting the test function $\varphi = [v^{p}]_{h} \eta^{2}$ into \eqref{eqn:wf2} and passing the limit $h \to 0$, 
we obtain
\[
\iint_{Q_{2}} |\nabla v|^{2} v^{p - 1} \, dx dt 
\le
\frac{C}{R^{2}}
\iint_{Q_{3}}
v^{p + 1}
\, dx dt.
\]
Hence, we obtain
\[
\iint_{Q_{2}} |\nabla v| \, dx dt
\le
\left(
\iint_{Q_{3}} v^{p + 1} \, dx dt 
\right)^{1 / 2}
\left(
\iint_{Q_{2}} v^{1 - p} \, dx dt 
\right)^{1 / 2}.
\]
Since $0 < 1 + p < 1 - p < (n + 2) / n$, 
Lemma \ref{lem:whi} yields \eqref{eqn:boundary_esti1}.

\textbf{Step 2.}
Let $u$ be a function in Lemma \ref{lem:boundary_regularity_esti}.
For $0 < r < R$, set 
\[
M(r) := \esssup_{Q( \Xi_{0}, r) \cap D} u,
\quad
m(r) := \essinf_{Q( \Xi_{0}, r) \cap D} u.
\]
Applying \eqref{eqn:boundary_esti1} to $M(4R) - u$, we get
\[
M(4r) - \sup_{\partial D \times (0, T) \cap Q_{1}(r)} u
\le
\frac{C}{\gamma}
\left(
M(4r) - M(r) + k(4r)
\right),
\]
where $Q_{1}(r) = B(\xi_{0}, r) \times (t_{0} - 12 r^{2}, t_{0} - 8 r^{2})$.
Similarly, for $u - m(4r)$, we obtain
\[
\inf_{\partial D \times (0, T) \cap Q_{1}(r)} u - m(4r)
\le
\frac{C}{\gamma}
\left(
m(r) - m(4r) + k(4r)
\right).
\]
Summing these yields
\[
\osc_{Q(\Xi_{0}, r)} u
\le
\left( 1 - \frac{\gamma}{C} \right)
\osc_{Q(\Xi_{0}, 4r)} u
+
\frac{\gamma}{C} \osc_{\partial D \times (0, T) \cap Q(\Xi_{0}, 4r)} u
+
2 k(4r).
\]
Using this inequality iteratively, we arrive at the desired estimate.
\end{proof}

\section*{Acknowledgments}
This work was supported by JSPS KAKENHI (doi:10.13039/501100001691) Grant Number 23H03798.
% 門松は冥土の旅の一里塚めでたくもありめでたくもなし

%%%%%%%%%%%%%%%%%%%%%%%%%%%%%%%%%%%%%%%%
% References
%%%%%%%%%%%%%%%%%%%%%%%%%%%%%%%%%%%%%%%%

\bibliographystyle{amsalpha} % plain, alpha, abbrv, unsrt
\bibliography{reference}
% \nocite{*}

%%%%%%%%%%%%%%%%%%%%%%%%%%%%%%%%%%%%%%%%
% The end of the document
%%%%%%%%%%%%%%%%%%%%%%%%%%%%%%%%%%%%%%%%

\end{document}